\documentclass[12pt, reqno]{amsart}
\usepackage{amscd,amsmath,amsthm,amssymb,graphics}
\usepackage{amsfonts,amssymb,amscd,amsmath,enumitem,verbatim}
\usepackage[a4paper,top=3cm,left=3cm,right=3cm]{geometry}
\theoremstyle{plain}
\usepackage{hyperref}
\newtheorem{Theorem}{Theorem}
\newtheorem{Lemma}[Theorem]{Lemma}

\newtheorem{Proposition}[Theorem]{Proposition}

\theoremstyle{definition}
\newtheorem{Definition}[Theorem]{Definition}
\newtheorem{Remark}[Theorem]{Remark}

\newtheorem{Example}[Theorem]{Example}

\usepackage{graphicx}

\usepackage{float}
\title{A new algorithm for $p$--adic continued fractions}

\begin{document}

\maketitle
\begin{center}
  NADIR MURRU\textsuperscript{*} AND GIULIANO ROMEO\textsuperscript{\dag}
  \vskip 0.4cm
  \textsuperscript{*} Department of Mathematics, University of Trento \par
  \textsuperscript{\dag} Department of Mathematical Sciences, Politecnico of Turin  \par
  \vskip0.3cm
  email: nadir.murru@unitn.it, giuliano.romeo@polito.it
\end{center}

\begin{abstract}
Continued fractions in the field of $p$--adic numbers have been recently studied by several authors. It is known that the real continued fraction of a positive quadratic irrational is eventually periodic (Lagrange's Theorem). It is still not known if a $p$--adic continued fraction algorithm exists that shares a similar property. In this paper we modify and improve one of Browkin's algorithms. This algorithm is considered one of the best at the present time. Our new algorithm shows better properties of periodicity. We show for the square root of integers that if our algorithm produces a periodic expansion, then this periodic expansion will have pre-period one. It appears experimentally that our algorithm produces more periodic continued fractions for quadratic irrationals than Browkin's algorithm. Hence, it is closer to an algorithm to which an analogue of Lagrange's Theorem would apply.
\end{abstract}

\section{Introduction}
Classical continued fractions are very important in number theory and they are mainly employed in the area of Diophantine approximation. In fact, the algorithm to compute the continued fraction expansion of a real number provides the best rational approximations (see, e.g., \cite{OLDS}). Moreover, a real continued fraction is finite if and only if it represents a rational number and it is periodic if and only if it represents a quadratic irrational (the Lagrange's Theorem). In 1940, Mahler \cite{MAH} started the study of continued fractions in the field of $p$--adic numbers $\mathbb Q_p$. 
Since then, a lot of research has been done in order to find an algorithm that replicates the above properties of real continued fractions. This is still an open problem since a $p$--adic analogue of the Lagrange's Theorem has not been proved yet. In $\mathbb{R}$, the algorithm for continued fractions is essentially derived from the Euclidean algorithm and thus the floor function is mainly involved. However, in $\mathbb{Q}_p$ there is no a canonical choice for a $p$--adic floor function, despite some studies about a $p$--adic version of the Euclidean algorithm (see, \cite{Err, Lager}). Consequently, there have been several attempts for defining a $p$--adic continued fractions algorithm, see \cite{BI, BII, RUB, SCH}.
Schneider's \cite{SCH} and Ruban's \cite{RUB} algorithms do not provide a finite expansion for all rational numbers \cite{BUN, HW, LAO}. Moreover, they do not provide a periodic continued fraction for all quadratic irrationals. A characterization of their periodicity can be found in \cite{TIL} and \cite{CVZ}. Browkin's algorithms \cite{BI, BII} are the most interesting since they stop in a finite number of steps if and only if the input is a rational number, similarly to continued fractions in $\mathbb{R}$. 

Given $\alpha=\sum\limits_{i=-r}^{+\infty}a_ip^i\in\mathbb{Q}_p$, with $a_i \in \{ -\frac{p-1}{2}, \ldots, \frac{p-1}{2} \}$, Browkin defined  two floor functions as
\[s(\alpha)=\sum\limits_{i=-r}^{0}a_ip^i, \ \ \ t(\alpha)=\sum\limits_{i=-r}^{-1}a_ip^i, \]
with $r \in \mathbb N$, where $s(\alpha)=0$ for $r<0$ and $t(\alpha)=0$ for $r\leq 0$.
Given $\alpha_0 \in \mathbb Q_p$, the first algorithm defined by Browkin (see \cite{BI}), which we will call \textit{Browkin I}, evaluates the $p$--adic continued fractions expansion $[b_0, b_1, \ldots]$ of $\alpha_0$ by 
\begin{align} 
\begin{cases}\label{Br1}
b_n=s(\alpha_n)\\
\alpha_{n+1}=\frac{1}{\alpha_n-b_n},
\end{cases}
\end{align}
for all $n\geq 0$.
In \cite{BII}, Browkin proposed a new algorithm, which we will call \textit{Browkin II}, where the $p$--adic continued fraction $[b_0, b_1, \ldots]$ of $\alpha_0 \in \mathbb Q_p$ is evaluated by

\begin{align} 
\begin{cases}\label{Br2}
b_n=s(\alpha_n) \ \ \ \ \ & \textup{if} \ n \ \textup{even}\\
b_n=t(\alpha_n) & \textup{if} \ n \ \textup{odd}\ \textup{and} \ v_p(\alpha_n-t(\alpha_n))= 0\\
b_n=t(\alpha_n)-sign(t(\alpha_n)) & \textup{if} \ n \ \textup{odd} \ \textup{and} \ v_p(\alpha_n-t(\alpha_n))\neq 0\\
\alpha_{n+1}=\frac{1}{\alpha_n-b_n}.
\end{cases}
\end{align}
for all $n\geq 0$.
In both algorithms, the choice of the representatives $a_i$'s in the interval $\{-\frac{p-1}{2},\ldots,\frac{p-1}{2}\}$ is crucial. 
An analogue of the Lagrange's Theorem has not been proved nor disproved for these two algorithms and an effective characterization of the periodicity is still missing. For \textit{Browkin I}, Bedocchi \cite{BEI,BEII,BEIII} characterized the purely periodic expansions. Furthermore, he gave some results about the possible lengths of the periods and pre-periods. Some similar results have been recently proved for \textit{Browkin II}, see \cite{MRSI}. In \cite{CMT}, the authors proved that there are infinitely many square roots of integers having a periodic \textit{Browkin I} continued fraction with period of length $2^k$, for each $k\geq 1$. Moreover, they gave some criteria for the periodicity of this algorithm, although they are not effective. Further results about the periodicity and the algebraic properties of $p$--adic continued fractions can be found in \cite{DEA,OO,WANGI,WANGII,DEWI,DEWII}.
\textit{Browkin II} appears to provide more periodic expansions for quadratic irrationals than \textit{Browkin I}, as experimentally observed by Browkin \cite{BII}. Hence, it is interesting to deepen the study of this algorithm in order to improve it furthermore. 
In \cite{BCMII}, an algorithm very similar to \emph{Browkin II} has been considered, using the canonical representatives in $\{0,\ldots,p-1\}$ instead of $\{-\frac{p-1}{2},\ldots,\frac{p-1}{2}\}$. However, experimentally, it did not seem to improve \textit{Browkin II} in terms of periodic expansions. 

In the first part of the paper we continue the analysis of the periodicity of \textit{Browkin II}. The presence of the sign function in \eqref{Br2} makes the study of the periodicity partial and unsatisfactory, in a sense that we explain in Section \ref{sec: bad}. In particular, there is not a characterization for the pure periodicity, in contrast with Galois' Theorem in $\mathbb{R}$ and the results of \cite{BEI} for \textit{Browkin I}. Moreover, the length of pre-periods of periodic \textit{Browkin II} continued fractions for square roots of integers can  assume several different values. On the contrary, in $\mathbb{R}$ it is always $1$ and for \textit{Browkin I} it is either $2$ or $3$ (see \cite{BEI}). These results are crucial for the study of the periodicity and for the proof of Lagrange's Theorem. In this paper we propose a new algorithm, which is obtained as a slight modification of \textit{Browkin II}, where the sign function is not used. 
In \textit{Browkin II}, the sign function is used in order to fulfill the convergence condition proved in \cite{BII}. This condition has been generalized very recently in \cite{MRSII}. In this way, the omission of the sign does not compromise the $p$--adic convergence (more details are given in Section \ref{sec: new}). This small adjustment improves significantly the periodicity properties of \textit{Browkin II}. In fact, for the new algorithm an analogue of Galois' Theorem holds (see Section \ref{secPURE}). Moreover, the pre-period of periodic continued fractions for square roots of integers with zero valuation has length $1$ (see Section \ref{sub: prepe}). 
In Section \ref{sec: numerical}, we make several computations for studying the behaviour of the new algorithm, compared with \textit{Browkin I} and \textit{Browkin II}. In particular, we notice that \textit{Browkin I} is periodic very rarely compared with \textit{Browkin II} and the new algorithm. Moreover, it appears that the new algorithm is usually periodic on more quadratic irrationals than \textit{Browkin II} (in all the cases except when $p=3$) and they have a similar behaviour for large values of the prime $p$ (see Remark \ref{Rem: exp}). We also propose an analysis about the quality of the approximations provided by the three algorithms and by truncating the $p$--adic expansion of square roots.
 The computations and the theoretical results suggest that the use of the sign function in \textit{Browkin II} affects negatively its periodicity and it can be removed in \eqref{Br2} without, apparently, any negative effect. On the contrary, removing the use of the sign function allows to study more easily the algorithms and to get more useful properties.

\section{Some issues of Browkin's algorithm} \label{sec: bad}
In the following, $p$ is an odd prime and we denote by $v_p(\cdot)$ and $|\cdot|_p$, respectively, the $p$--adic valuation and the $p$--adic norm. We also call $J_p$ the set of all the possible values taken by the function $s$ and the set $K_p$ is the analogue of $J_p$ for the function $t$. In the following proposition, we summarize some well known facts.
\begin{Proposition} \label{prop:varie} \
\begin{enumerate}
    \item For all $a,b\in J_p$, with $a\neq b$, we have $v_p(a-b)\leq 0$, see \cite{BEI}.
    \item For all $a,b\in K_p$, with $a\neq b$, we have $v_p(a-b)< 0$, see \cite{MRSI}.
    \item Let $\alpha\in\mathbb{Q}_p$, then
\[|s(\alpha)|<\frac{p}{2} \quad \text{and} \quad |t(\alpha)|<\frac{1}{2}, \]
where $|\cdot|$ is the Euclidean norm, see \cite{BI, MRSI}.
\end{enumerate}
\end{Proposition}

The first issue about \textit{Browkin II} arises when studying the pure periodicity.

Purely periodic continued fractions have always been of great interest and they are crucial in the proof of the Lagrange's Theorem for classical continued fractions. The characterization of purely periodic continued fractions in $\mathbb{R}$ is a famous result due to Galois.
For \textit{Browkin I}, Bedocchi proved the following theorem.

\begin{Theorem}[\cite{BEI}, Proposition 3.1, Proposition 3.2]
Let $\alpha\in\mathbb{Q}_p$ having a periodic \textit{Browkin I} continued fraction expansion. Then the expansion is purely periodic if and only if
\[|\alpha|_p> 1, \quad |\overline{\alpha}|_p<1.\]
\end{Theorem}

\begin{Remark}
Let us recall that a periodic Browkin's continued fraction represents always an irrational $\alpha\in\mathbb{Q}_p$ quadratic over $\mathbb{Q}$, i.e., its minimal polynomial is an irreducible polynomial $f(x)\in\mathbb{Q}[x]$ of degree $2$. We denote by $\overline{\alpha}$ the conjugate of $\alpha$, i.e., the other root of $f(x)$ over $\mathbb{Q}$.
\end{Remark}

Very recently, a similar result was proved in \cite{MRSI} for \textit{Browkin II}. However, in this case the result is only partial.

\begin{Theorem}[\cite{MRSI}]
Let $\alpha\in\mathbb{Q}_p$ having a periodic \textit{Browkin II} continued fraction expansion. If the expansion is purely periodic, then
\[ |\alpha|_p=1, \ \ |\overline{\alpha}|_p<1.\]
Vice versa, if
\[ |\alpha|_p=1, \ \ |\overline{\alpha}|_p<1,\]
then the pre-period of the continued fraction expansion of $\alpha$ must have even length.
\end{Theorem}

As already pointed out in \cite{MRSI}, the result can not be strengthened. For example, the $7$--adic expansion of $\alpha=\sqrt{30}+3$ is
\[ \alpha=\left[-1, \frac{3}{7}, 3, \frac{2}{7},\overline{1, \frac{2}{7}, -2, \frac{3}{7}, 1, \frac{2}{7}, 2, \frac{1}{7}, -1, -\frac{5}{7}}\right],\]
that is not purely periodic although $|\alpha|_p=1$ and $|\overline{\alpha}|_p<1$. The problem in the proof arises from the presence of the sign function in \textit{Browkin II}, hence there is not a straightforward characterization as in the other cases.

The second issue of \textit{Browkin II} regards the possible pre-period lengths for periodic continued fractions of the square roots of integers. A known result for continued fraction in $\mathbb{R}$ states that, if $D\in\mathbb{N}$ is a non-square integer, then, for some integers $b_0,\ldots,b_{k-1}$, we have
\[\sqrt{D}=\left[b_0,\overline{b_1,\ldots,b_{k-1},2b_0}\right].\]
Hence, in particular, every square root of integer has a periodic continued fraction of pre-period $1$. For \textit{Browkin I}, Bedocchi proved the following result.
\begin{Theorem}[\cite{BEI}, Proposition 3.3]
Let $D\in\mathbb{Z}$ such that $\sqrt{D}\in\mathbb{Q}_p$. If the \textit{Browkin I} expansion of $\sqrt{D}$ is periodic, then the pre-period is
\begin{align*} 
\begin{cases}
2 \ \ \ \ \ & D \not\equiv 4 \mod 8 \ \text{when} \ p=2 \\
3 \ \ \ \ \ & otherwise.
\end{cases}
\end{align*}
\end{Theorem}

The situation is more complicated for the case of \textit{Browkin II}, where several different periods have been observed for $\sqrt{D}$. In \cite{MRSI}, the following result has been proved, showing that in this case the pre-period can not be an odd integer greater than $1$.

\begin{Proposition}[\cite{MRSI}]\label{prepMRS}
Let $\sqrt{D}$ be defined in $\mathbb{Q}_p$, with $D\in\mathbb{Z}$ not a square; if $\sqrt{D}$ has  a periodic continued fraction expansion with \textit{Browkin II}, then the pre-period has length either $1$ or even.
\end{Proposition}

The problems that we have highlighted for \textit{Browkin II} are due to the (unforeseeable) use of the sign function. In fact, the properties of periodicity of \textit{Browkin II} strongly depend on the application of the sign function during the algorithm. 

\begin{Definition}\label{signBR}

Given $\alpha = \sum\limits_{n=-r}^{+\infty} a_np^n \in \mathbb Q_p$, let us define the function $B:\mathbb{Q}_p\rightarrow \{-1,0,+1\}$ as follows:
\begin{align} \label{new-alg}
B(\alpha):=
\begin{cases} 
-1 &\textup{if}  \ a_0=0 \ \text{and} \ sign(t(\alpha))=-1 \\
0 \ \  &\textup{if}  \ a_0\neq 0\\
+1 &\textup{if}  \ a_0=0 \ \text{and} \ sign(t(\alpha))=+1.
\end{cases}
\end{align}
\end{Definition}

Using this definition, we can rewrite Algorithm \eqref{Br2} as
\begin{align*} 
\begin{cases}\label{Br2star}
b_n=s(\alpha_n) \ \ \ \ \ & \textup{if} \ n \ \textup{even}\\
b_n=t(\alpha_n)-B(\alpha_n) & \textup{if} \ n \ \textup{odd}\\
\alpha_{n+1}=\frac{1}{\alpha_n-b_n},
\end{cases}
\end{align*}
for all $n\geq0$ and given $\alpha_0 \in \mathbb Q_p$.
Notice that we are interested on the value of $B(\alpha_n)$ for odd $n$, i.e., when we use the function $t$. The problem of determining the exact behaviour of $B(\alpha_n)$ seems to be hard in general but it is crucial for the study of the periodicity of \textit{Browkin II}.

In the next theorem, we prove a necessary and sufficient condition to decide whether or not the sign function is going to be used at the $(k+1)$-th step only looking at the coefficients of the $p$--adic expansion of the $k$-th complete quotient. The effectiveness of this result is that it does not require the explicit computation of $\alpha_{k+1}$. Before stating the theorem, we need the following definition.

\begin{Definition}
Given $\alpha = \sum\limits_{i=-r}^{+\infty}c_ip^i \in \mathbb Q_p$, we define the matrix
\[C_\alpha=\begin{pmatrix}
\ c_{n+1} & c_n & 0  & \ldots &   0 \ \\
\ c_{n+2} & c_{n+1} & c_n & \ddots  & \vdots  \  \\
\ \vdots & \vdots & \ddots  & \ddots &  0 \  \\
\ c_{2n-1} & c_{2n-2} & \ldots   & \ddots & c_n \ \\
\ c_{2n} & c_{2n-1} & \ldots & \ldots & c_{n+1} \ \\
\end{pmatrix},\]
where $n=v_p(\alpha-s(\alpha))$.
\end{Definition}

\begin{Theorem}
Given $\alpha_0\in\mathbb{Q}_p$, for all even $k\in\mathbb{N}$, $B(\alpha_{k+1})\neq 0$ if and only if $det(C_{\alpha_k})=0$, where $\alpha_k$ is the complete quotients obtained by Algorithm \eqref{Br2}.
\begin{proof}
Let $\alpha_k = \sum\limits_{i=-r'}^{+\infty}c_ip^i$ and $\alpha_{k+1} = \sum\limits_{i=-r}^{+\infty}a_ip^i$ be two consecutive complete quotients obtained by \eqref{Br2} with $k$ even.
By definition $B(\alpha_{k+1})\neq 0$ if and only if $a_0=0$.
We call $n=v_p(\alpha_k-s(\alpha_k))>0$, so that 
\[\alpha_k-s(\alpha_k)=c_np^n+c_{n+1}p^{n+1}+\ldots.\]
In this case the valuation of $\alpha_{k+1}$ is
\[v_p(\alpha_{k+1})=-v_p(\alpha_{k}-s(\alpha_{k}))=-n,\]
hence we can write it as
\[\alpha_{k+1}=a_{-n}\frac{1}{p^n}+a_{-(n-1)}\frac{1}{p^{n-1}}+\ldots+a_{-1}\frac{1}{p}+a_0+\ldots.\]
We want that $\alpha_{k+1}(\alpha_{k}-s(\alpha_{k}))=1$, that is,
\[c_na_{-n}+(c_{n+1}a_{-n}+c_{n}a_{-(n-1)})p+\ldots+(c_{2n}a_{-n}+\ldots+c_na_0)p^n+\ldots=1.\]
Hence, the coefficients $a_i$ are uniquely determined as solutions of the following system:
\begin{align*}
\begin{cases}
a_{-n}c_n=1\\
a_{-n}c_{n+1}+a_{-(n-1)}c_n=0\\
a_{-n}c_{n+2}+a_{-(n-1)}c_{n+1}+a_{-(n-2)}c_{n}=0\\
\ldots\\
\sum\limits_{k=0}^{n} a_{-n+k}c_{2n-k}=0.
\end{cases}
\end{align*}

If we call
\[C=\begin{pmatrix}
\ c_n &   0 & 0 & \ldots &  0 \  \\
\ c_{n+1} & c_n & 0  & \ldots & 0 \  \\
\ c_{n+2} & c_{n+1} & c_n & \ddots  & \vdots  \  \\
\ \vdots & \vdots & \ddots & \ddots & 0  \ \\
\ c_{2n} & c_{2n-1} & \ldots & c_{n+1} & c_{n} \ \\
\end{pmatrix}\]
the $(n+1)\times(n+1)$ matrix of the coefficients, then $a_0$ is
\[a_0=\frac{det\begin{pmatrix}
\ c_n &   0 & 0 & \ldots &  1 \  \\
\ c_{n+1} & c_n & 0  & \ldots & 0 \  \\
\ c_{n+2} & c_{n+1} & c_n & \ddots  & \vdots  \  \\
\ \vdots & \vdots & \ddots & \ddots & 0  \ \\
\ c_{2n} & c_{2n-1} & \ldots & c_{n+1} & 0 \ \\
\end{pmatrix}}{det(C)}.\]
In particular, $a_0=0$ if and only if the numerator is zero, that is $det(C_{\alpha_k})=0$.
\end{proof}
\end{Theorem}
Although it is possible to predict the appearance of the sign function one step in advance, it seems difficult to generalize this construction to make a prediction at the generic step. Therefore, the use of the sign function makes incomplete the study of the periodicity.

For all the reasons highlighted in this section, for the rest of the paper we focus on a new algorithm that is very similar to \textit{Browkin II}, but its properties do not rely on the function $B$.

\section{The new algorithm}\label{sec: new}
The algorithm \textit{Browkin II} has been defined exploiting the use of the sign function in order to satisfy the hypothesis of the following lemma. 

\begin{Lemma}[\cite{BII}, Lemma 1]\label{ConvBr2}
Let $b_0,b_1,\ldots\in \mathbb{Z}\left[\frac{1}{p}\right]$ be an infinite sequence  such that
\begin{equation}
\begin{cases}
v_p(b_{2n})=0\\
v_p(b_{2n+1})<0,
\end{cases}
\end{equation}
for all $n > 0$.
Then the continued fraction $[b_0,b_1,\ldots]$ is convergent to a $p$--adic number.
\end{Lemma}
In \textit{Browkin II}, the sign function is necessary in order to generate the even partial quotients always with zero valuations. Otherwise it might happen that
\[v_p(\alpha_n-t(\alpha_n))>0,\]
for some odd $n\in\mathbb{N}$, which implies
\[v_p(b_{n+1})=v_p(s(\alpha_{n+1}))<0,\]
with $n+1$ even, and the hypotheses of Lemma \ref{ConvBr2} are not satisfied.

However, in \cite{MRSII}, this convergence condition has been lightened, proving the following sufficient condition for the $p$--adic convergence of continued fractions.

\begin{Lemma}[\cite{MRSII}]
Let  $b_0,b_1,\ldots\in \mathbb{Z}\left[\frac{1}{p}\right]$ be an infinite sequence such that
\[v_p(b_{n}b_{n+1})<0,\]
for all $n > 0$.
Then the continued fraction $[b_0,b_1,\ldots]$ is convergent to a $p$--adic number.
\end{Lemma}

Therefore, it is sufficient to have one partial quotient with negative valuation for each two partial quotients, and the other one can have either negative or null valuation. Starting from these observations, we define the following algorithm. This algorithm is basically \textit{Browkin II} without the use of the sign function. In this way, we solve the issues on the periodicity that we have underlined in the previous section. Given $\alpha_0=\alpha$ it works as follows, for $n\geq 0$:
\begin{align} 
\begin{cases}\label{new}
b_n=s(\alpha_n) \ \ \ \ \ & \textup{if} \ n \ \textup{even}\\
b_n=t(\alpha_n) & \textup{if} \ n \ \textup{odd}\\
\alpha_{n+1}=\frac{1}{\alpha_n-b_n}.
\end{cases}
\end{align}

In the next section we see that this new algorithm does not present the issues of Browkin's second algorithm and, at the same time, does not lose its good properties in terms of finiteness and periodicity.

\subsection{Finiteness} \label{sub: fini}
In this section we prove that finite continued fractions produced by Algorithm \eqref{new} characterize rational numbers. The fact that every finite continued fraction represents a rational number is straightforward. In the following theorem we prove that the converse also holds. 
\begin{Theorem}
If $\alpha_0 \in \mathbb Q$, then Algorithm \eqref{new}  stops in a finite number of steps.
\begin{proof}
By the construction we have that, for all $n > 0$,
\begin{align*}
v_p(\alpha_{2n})& \leq 0,\\
v_p(\alpha_{2n+1})& <0,\\
v_p(\alpha_{2n+2})& \leq 0.
\end{align*}
Hence, the complete quotients have the form
\begin{align*}
\alpha_{2n}&=\frac{N_{2n}}{D_{2n}p^j}, & \ &\text{with} \ (N_{2n},D_{2n})=1, \ \ p\not| N_{2n}D_{2n},  \ \ j\geq 0,\\
\alpha_{2n+1}&=\frac{N_{2n+1}}{D_{2n+1}p^k}, & \ &\text{with} \ (N_{2n+1},D_{2n+1})=1, \ \ p\not| N_{2n+1}D_{2n+1},  \ \ k\geq 1,\\
\alpha_{2n+2}&=\frac{N_{2n+2}}{D_{2n+2}p^l}, & \ &\text{with} \ (N_{2n+2},D_{2n+2})=1, \ \ p\not| N_{2n+2}D_{2n+2},  \ \ l\geq 0.
\end{align*}
By Proposition \ref{prop:varie}, the partial quotients satisfy, for all $n\in\mathbb{N}$,
\[
|b_{2n}|=\left|\frac{c_{2n}}{p^j} \right|<\frac{p}{2},\ \ \ \ \ 
|b_{2n+1}|= \left|\frac{c_{2n+1}}{p^k} \right|<\frac{1}{2},\]
for some $j\geq 0$, $k\geq 1$ and $v_p(c_{2n})=v_p(c_{2n+1})=0$.
Using the formula $\alpha_{k+1}=\frac{1}{\alpha_k-b_k}$ for $k=2n,2n+1$, we obtain
\begin{align*}
N_{2n+1}(N_{2n}-c_{2n}D_{2n})&=p^{j+k}D_{2n+1}D_{2n},\\
N_{2n+2}(N_{2n+1}-c_{2n+1}D_{2n+1})&=p^{k+l}D_{2n+2}D_{2n+1}.
\end{align*}
Since $(N_n,pD_n)=1$ for all $n\in\mathbb{N}$, then
\[|N_{2n+2}|=|D_{2n+1}|, \ \ \ |N_{2n+1}|=|D_{2n}|, \]
and
\begin{align*}
|D_{2n+1}|p^{k+j}&=|N_{2n}-c_{2n}D_{2n}|,\\ |D_{2n+2}|p^{k+l}&=|N_{2n+1}-c_{2n+1}D_{2n+1}|.
\end{align*}
In fact, $|N_{2n+1}|$ divides $|p^{j+k}D_{2n+1}D_{2n}|$ but is coprime with $|p^{j+k}D_{2n+1}|$, so it must be equal to $|D_{2n}|$, and a similar argument is used for $|N_{2n+2}|$. Therefore,
\begin{align*}
|D_{2n+1}|&\leq\frac{|N_{2n}|+|c_{2n}||D_{2n}|}{p^{k+j}}= \frac{|N_{2n}|}{p^{k+j}}+\frac{|c_{2n}|}{p^j}\cdot\frac{1}{p^k}|D_{2n}| <\frac{|N_{2n}|}{p^{k+j}}+\frac{1}{2p^{k-1}}|D_{2n}|,\\
|D_{2n+2}|&\leq\frac{|N_{2n+1}|+|c_{2n+1}||D_{2n+1}|}{p^{k+l}}= \frac{|N_{2n+1}|}{p^{k+l}}+\frac{|c_{2n+1}|}{p^k}\cdot\frac{1}{p^l}|D_{2n+1}|<\\
&<\frac{|N_{2n+1}|}{p^{k+l}}+\frac{1}{2}\cdot\frac{1}{p^{l}}|D_{2n+1}|=\frac{|N_{2n+1}|}{p^{k+l}}+\frac{1}{2p^{l}}|D_{2n+1}|,
\end{align*}
so that, since $k\geq 1$ and $j,l\geq 0$,
\[|D_{2n+1}|<\frac{|N_{2n}|}{p}+\frac{|D_{2n}|}{2}, \ \ \ \ |D_{2n+2}|<\frac{|N_{2n+1}|}{p}+\frac{|D_{2n+1}|}{2}.\]

By using the above formulas we can write
 
\begin{align*}
|N_{2n+2}|+p|D_{2n+2}|&<|D_{2n+1}|+p\left(\frac{|N_{2n+1}|}{p}+\frac{|D_{2n+1}|}{2}\right)=\\
&=|N_{2n+1}|+\left(\frac{p}{2}+1\right) |D_{2n+1}|<\\
&<|D_{2n}|+\left(\frac{p}{2}+1\right) \left(\frac{|N_{2n}|}{p}+\frac{|D_{2n}|}{2}\right)=\\
&=\left(\frac{1}{p}+\frac{1}{2}\right)|N_{2n}|+\left(\frac{p}{4}+\frac{1}{2}+1 \right)|D_{2n}|.
\end{align*}

The coefficient of $|N_{2n}|$ is clearly less than or equal $1$ and the coefficient of $|D_{2n}|$ is less than or equal $p$ if and only if $p\geq\frac{p+6}{4}$, that is, if and only if $p\geq 2$. We can conclude that, for all $n\in\mathbb{N}$,
\[|N_{2n+2}|+p|D_{2n+2}|<|N_{2n}|+p|D_{2n}|.\]
The sequence $\{|N_{2n}|+p|D_{2n}|\}_{n\in\mathbb{N}}$ is then a strictly decreasing sequence of natural numbers and, hence, it is finite. Therefore $\alpha$ has a finite continued fraction and the thesis follows.    
\end{proof}
\end{Theorem}

\begin{Example}
Let us consider the continued fraction of $a=-\frac{17}{29}$ and $b=-\frac{15}{109}$ in $\mathbb{Q}_{23}$. For $a$, the $23$--adic expansion is
\[a=\left[1,-\frac{3}{23},-2 \right],\]
that is equal for both algorithms. Instead, the expansion of $b$ is different, because with \textit{Browkin II} it is
\[b=\left[-9,\frac{13}{23},-1,\frac{4}{23},-1 \right],\]
while the result with Algorithm \eqref{new} is
\[b=\left[-9,-\frac{10}{23},\frac{42}{23}\right].\]
The fact that the expansion of rational numbers with Algorithm \eqref{new} is shorter than the expansion with \textit{Browkin II} appears to be very frequent.
\end{Example}

\subsection{Periodicity} \label{secPURE}
In this section we provide the characterization of purely periodic continued fractions obtained with Algorithm \eqref{new}. 

\begin{Theorem}\label{pureR}
If $\alpha\in\mathbb{Q}_p$ has a periodic continued fraction expansion by means of Algorithm \eqref{new}, then the expansion is purely periodic if and only if
\[|\alpha|_p\geq 1, \ \ \ \ |\overline{\alpha}|_p<1. \]
\end{Theorem}
\begin{proof}
By the pure periodicity, we get that
\[|\alpha|_p=|b_0|_p=|b_k|_p\geq 1,\]
since the length $k$ of the period is even. Indeed, if $k$ is odd we can consider $2k$, so that the periodic part always starts by using the same floor function.
\begin{comment}
Indeed, as in \textit{Browkin II}, we can not have periodic continued fractions of odd period length since we use different floor functions.
\end{comment}
For the norm of the conjugate $\overline{\alpha}$ we apply the usual relation (see \cite{BEI} and \cite{MRSI}), that is
\[|\overline{\alpha}|_p=\frac{1}{|b_{k-1}|_p}.\]
Since $k-1$ is odd, this implies that $|b_{k-1}|_p>1$ and $|\overline{\alpha}|_p<1$, proving the necessary condition for the pure periodicity. Conversely, let us consider a periodic continued fraction expansion
\[\alpha=\left[b_0,\ldots,b_{h-1},\overline{b_h,\ldots,b_{h+k-1}}\right],\]
and let us assume that $|\alpha|_p\geq 1$ and $|\overline{\alpha}|_p<1$. We are going to prove that $h=0$, namely the expansion is purely periodic. For all $n\in \mathbb{N}$, the valuations of the complete quotients are
\begin{align*}
v_p(\alpha_{2n})&=v_p(b_{2n})\leq 0,\\
v_p(\alpha_{2n+1})&=v_p(b_{2n+1})< 0.
\end{align*}
For the valuations of their conjugates, let us notice that $|\overline{\alpha}_0|=|\overline{\alpha}|_p<1$ means $v_p(\overline{\alpha}_0)>0$, so that:
\begin{align*}
v_p(\overline{\alpha}_1)&=v_p\Big( \frac{1}{\overline{\alpha}_0  - b_0}\Big)=-v_p(\overline{\alpha}_0  - b_0)=-v_p(b_0)\geq0,\\
v_p(\overline{\alpha}_2)&=v_p\Big( \frac{1}{\overline{\alpha}_1  - b_1}\Big)=-v_p(\overline{\alpha}_1  - b_1)=-v_p(b_1)>0.
\end{align*}
The last two inequalities are true since
\begin{align*}
    v_p(\overline{\alpha}_0)&>0, \ \  & v_p(b_0)&\leq 0,\\
    v_p(\overline{\alpha}_1)&\geq 0, \ \  & v_p(b_1)&< 0,
\end{align*}
hence $v_p(\overline{\alpha}_i)<v_p(b_i)$ for both $i=1,2$. This argument can be iterated, so that, for all $n\in\mathbb{N}$,
\begin{align*}
v_p(\overline{\alpha}_{2n+1})&=v_p\Big( \frac{1}{\overline{\alpha}_{2n}  - b_{2n}}\Big)=-v_p(\overline{\alpha}_{2n}  - b_{2n})=-v_p(b_{2n})\geq 0,\\
v_p(\overline{\alpha}_{2n+2})&=v_p\Big( \frac{1}{\overline{\alpha}_{2n+1}  - b_{2n+1}}\Big)=-v_p(\overline{\alpha}_{2n+1}  - b_{2n+1})=-v_p(b_{2n+1})> 0.
\end{align*}

By the periodicity of the continued fraction of $\alpha$, we can observe that
\[\frac{1}{\alpha_{h-1}-b_{h-1}}=\alpha_h=\alpha_{h+k}=\frac{1}{\alpha_{h+k-1}-b_{h+k-1}}.\]
Therefore, we easily obtain the two relations
\begin{align*}
v_p(\alpha_{h-1}-\alpha_{h+k-1})&=v_p(b_{h-1}-b_{h+k-1}),\\
v_p(\overline{\alpha}_{h-1}-\overline{\alpha}_{h+k-1})&=v_p(b_{h-1}-b_{h+k-1}).
\end{align*}
Now, let us assume by contradiction that $h>0$ odd is the minimal starting point of the periodicity. Since in this case $h-1$ and $h+k-1$ are even,
\[v_p(\overline{\alpha}_{h-1})>0 \ \ \text{and} \ \ v_p(\overline{\alpha}_{h+k-1})>0.\]
It follows that
\[v_p(b_{h-1}-b_{h+k-1})=v_p(\overline{\alpha}_{h-1}-\overline{\alpha}_{h+k-1})\geq \min \{v_p(\overline{\alpha}_{h-1}),v_p(\overline{\alpha}_{h+k-1}) \}>0.\]
Since the partial quotients $b_{h-1}$ and $b_{h+k-1}$ are generated with the function $s$, by Proposition \ref{prop:varie} we conclude that $b_{h-1}=b_{h+k-1}$, that is, the periodicity can start at $h-1$. Therefore, the pre-period can not be odd.\\
Let us now assume that $h\geq 2$ is even.
In this case $h-1$ and $k-1$ are odd, hence
\[v_p(\overline{\alpha}_{h-1})\geq 0, \ \ \ v_p(\overline{\alpha}_{h+k-1})\geq 0.\]
Reasoning as in the previous case, we obtain
\[v_p(b_{h-1}-b_{h+k-1})_p=v_p(\overline{\alpha}_{h-1}-\overline{\alpha}_{h+k-1})\geq \min \{v_p(\overline{\alpha}_{h-1}),v_p(\overline{\alpha}_{h+k-1}) \}\geq 0.\]
In this second case, $b_{h-1}$ and $b_{h+k-1}$ are generated with the function $t$, hence by Proposition \ref{prop:varie} we conclude that $b_{h-1}=b_{h+k-1}$. Thus the pre-period can not be a positive even number. It follows that $h=0$ and the expansion of $\alpha$ with Algorithm \eqref{new} is purely periodic.
\end{proof}

\begin{Remark}
In \cite{BCMI}, the authors proved some conditions to obtain periodic expansions, with pre-period length $1$ and period length $2$, for quadratic irrationals using \textit{Browkin II}. Moreover, in \cite{BII} and \cite{MRSI}, it has been proved that there exist infinitely many square roots of integers that are periodic with period length $2$ and $4$. In all these cases, the sign function is not used during the algorithm, hence these results are true also for Algorithm \eqref{new}.
\end{Remark}

\subsection{Pre-periods for expansions of square roots}\label{sub: prepe}
In this section we analyze the possible lengths of pre-periods for the expansions of square roots of integers obtained using Algorithm \eqref{new}. In particular, we show that the pre-period length is always $1$ for square roots of integers with valuation zero, obtaining a result similar to Galois' Theorem.

In order to prove it, we define the following algorithm, which is similar to Algorithm \eqref{new} but the role of the functions $s$ and $t$ is switched.

\begin{align} \label{newstar}
\begin{cases}
b_n=t(\alpha_n) \ \ \ \ \ & \textup{if} \ n \ \textup{even}\\
b_n=s(\alpha_n) & \textup{if} \ n \ \textup{odd}\\
\alpha_{n+1}=\frac{1}{\alpha_n-b_n}.
\end{cases}
\end{align}

The $p$--adic convergence of the continued fraction generated by this algorithm is guaranteed since also in this case $v_p(b_nb_{n+1})<0$ for all $n\in\mathbb{N}$ (see \cite{MRSII}). In order to characterize the length of the pre-periods for Algorithm \eqref{new} we need an analogue of Theorem \ref{pureR} for Algorithm \eqref{newstar}.

\begin{Theorem}\label{pureRstar}
Let us consider $\alpha\in\mathbb{Q}_p$ with a periodic continued fraction obtained using Algorithm \eqref{newstar}. Then the expansion is purely periodic if and only if
\[|\alpha|_p> 1, \ \ \ \ |\overline{\alpha}|_p\leq 1. \]
\end{Theorem}
\begin{proof}
The proof is straightforward adapting the technique of Theorem \ref{pureR} and switching the two floor functions.
\end{proof}

Using the results of Theorem \ref{pureR} and \ref{pureRstar} we are able to prove the following result, characterizing the pre-period of periodic continued fraction expansions of square roots of integers obtained using Algorithm \eqref{new}.

\begin{Proposition}\label{newprep}
Let $\sqrt{D}\in\mathbb{Q}_p$, with $D\in\mathbb{Z}$ not a square and $p\nmid D$, having a periodic continued fraction obtained with Algorithm \eqref{new}. Then the pre-period has length $1$.
\end{Proposition}
\begin{proof}
By the characterization of Theorem \ref{pureR}, $\alpha$ can not have a purely periodic continued fraction. We can write $\alpha$ as
\[\alpha=b_0+\frac{1}{\alpha_1},\]
where $b_0=s(\alpha)$ and $\alpha_1$ is the second complete quotient.
In order to prove that the periodic expansion of $\alpha$ has pre-period of length $1$ we show that $\alpha_1$ has purely periodic expansion with Algorithm \eqref{newstar} starting with the function $t$. Therefore, by Theorem \ref{pureRstar}, we want to prove that $v_p(\alpha_1)<0$ and $v_p(\overline{\alpha}_1)\geq 0$. First of all we notice that, since $\alpha$ has a periodic continued fraction expansion, also $\alpha_1$ does. Then, by the construction of the algorithm,
\[v_p(\alpha_1)=-v_p(\alpha-s(\alpha))<0,\]
hence the condition on $\alpha_1$ is true. Since $\overline{\alpha}=-\sqrt{D}$, then
\[\overline{\alpha}=-b_0+a_1'p+a_1'p^2+\ldots.\]
We have that
\[v_p(\overline{\alpha}_1)=-v_p(\overline{\alpha}-b_0)<0,\]
if and only if $v_p(-2b_0+\ldots)>0$, that is never the case for $p\neq 2$. This means that $v_p(\overline{\alpha}_1)\geq 0$ and, by Theorem \ref{pureR}, it has a purely periodic expansion
\[\alpha_1=\left[\overline{b_1,\ldots,b_{k}}\right].\]
Therefore, the expansion of $\alpha=\sqrt{D}$ is
\[\sqrt{D}=\left[b_0,\overline{b_1,\ldots,b_{k}}\right],\]
that has pre-period $1$.
\end{proof}

To conclude this section, we consider also the case $\alpha=\sqrt{D}$, with $v_p(\sqrt{D})\neq 0$. If $v_p(\sqrt{D})<0$, then
\[\alpha_{1}=\frac{1}{\alpha-b_0}, \ \ \ \overline{\alpha}_{1}=\frac{1}{\overline{\alpha}-b_0}, \]
so that $v_p(\alpha_{1})<0$ and $v_p(\overline{\alpha}_{1})>0$. Hence, with a similar argument of Proposition \ref{newprep}, also in this case we can conclude that $\alpha_{1}$ is purely periodic and the continued fraction of $\alpha$ has pre-period $1$. Notice that, if $v_p(\sqrt{D})<0$, then $D$ is not an integer but a rational whose denominator is divided by $p$. Instead, if $v_p(\sqrt{D})>0$, then $b_0=0$ and
\[\alpha_{1}=\frac{1}{\alpha}, \ \ \ \overline{\alpha}_{1}=\frac{1}{\overline{\alpha}}.\]
Since $v_p(\alpha_1)=v_p(\overline{\alpha}_1)<0$, we are exactly in the previous case, so the expansion of $\alpha$ has pre-period $2$ with a $0$ as first partial quotient. Hence, we have proved the following result, very similar to Galois' Theorem for classical continued fractions.

\begin{Proposition}\label{newcomplete}
Let $\sqrt{D}\in\mathbb{Q}_p$, with $D\in\mathbb{Z}$ and $v_p(\sqrt{D})=e$, having a periodic continued fraction using Algorithm \eqref{new}. Then the expansion of $\sqrt{D}$ has pre-period
\[\begin{cases}
1 \ \ \ \text{if } e\leq 0\\
2 \ \ \ \text{if } e>0.
\end{cases}\]
Moreover, in the case $e>0$, its continued fraction expansion is
\[\sqrt{D}=[0,b_0,\overline{b_1,\ldots,b_h}],\]
where $b_0,\ldots,b_h$ are the partial quotients of
\[\frac{\sqrt{D}}{D}=[b_0,\overline{b_1,\ldots,b_h}].\]
\end{Proposition}

\section{Numerical computations} \label{sec: numerical}
In this section we collect some numerical results about the Browkin's algorithms \eqref{Br1} and \eqref{Br2} and Algorithm \eqref{new}. 
It turns out that, in addition to the good theoretical results already highlighted in the previous sections, Algorithm \eqref{new} appears to be periodic on more quadratic irrationals than  \textit{Browkin I} and \textit{Browkin II} for all odd primes less than $100$, except $p=3$ for \textit{Browkin II}. 
All the computations have been performed on the first $1000$ complete quotients of $\sqrt{D}\in\mathbb{Q}_p$, for all the odd primes $p$ less than $100$ and $1\leq D\leq 1000$, with $D$ not a square and $v_p(D)=0$. The numerical computations have been performed using SageMath and the code is publicly available \footnote{\href{https://github.com/giulianoromeont/p-adic-continued-fractions}{https://github.com/giulianoromeont/p-adic-continued-fractions}}.

\subsection{Some \textit{Browkin II} expansions}
We start with an observation on the Euclidean norm of odd partial quotients that allows us to correct some wrong \textit{Browkin II} expansions listed in \cite{BII}.
\begin{Remark}
The Euclidean norm of the odd partial quotients in \textit{Browkin II}  can not be greater than $1$. In fact, we know by Proposition \ref{prop:varie} that $0\leq|t(\alpha_k)|<\frac{1}{2}$ for all $k\in\mathbb{N}$. Then, if the sign is not used, $|b_{k}|=|t(\alpha_k)|<\frac{1}{2}$. If the sign is used, then $|b_k|=1-|t(\alpha_{k})|$ in both cases, so $0<|b_k|<1$, for $k$ odd. A more precise result is given in \cite{BCMI}, where this inequality is determinant in the proof for the finiteness of \textit{Browkin II} over the rational numbers.
\end{Remark}
Starting from this remark, we can notice that some of the expansions listed in \cite{BII} are wrong. In fact, in the $5$--adic continued fraction of $\sqrt{34}$, $\sqrt{39}$, $\sqrt{54}$, $\sqrt{69}$ and $\sqrt{99}$, respectively, $b_5=-\frac{28}{25}$, $b_9=-\frac{6}{5}$, $b_9=\frac{28}{25}$, $b_3=\frac{6}{5}$ and $b_{13}=\frac{32}{25}$, that are all greater than $1$ in Euclidean norm. We believe it is an error of implementation where the sign function is added instead of being subtracted from $t(\alpha_k)$. Three of them are still periodic and the correct expansions are
\begin{align*}
\sqrt{34}=& \left[2,\overline{-\frac{1}{5},1,-\frac{2}{5},-1,\frac{22}{25},-1,-\frac{2}{5},1,-\frac{1}{5},-1,-\frac{6}{25},-1} \right],\\
\sqrt{54}=& \left[2,\frac{2}{25},-1,\frac{1}{5},-2,-\frac{1}{5},\overline{-1,-\frac{2}{5},1,\frac{4}{5}} \right],\\
\sqrt{69}=& \Bigg[2,\overline{-\frac{2}{5},1,-\frac{4}{5},1,-\frac{1}{5},-1,-\frac{2}{5},2,-\frac{1}{5},2,-\frac{12}{25},2,-\frac{1}{5},2,}\\
&\overline{-\frac{2}{5},-1,-\frac{1}{5},1,-\frac{4}{5},1,-\frac{2}{5},-1,\frac{2}{5},1,-\frac{76}{125},1,\frac{2}{5},-1} \Bigg],
\end{align*}
while for $\sqrt{39}$ and $\sqrt{99}$ we have not observed any period.

\subsection{Periodic square roots of integers}
In this section we collect the results on the periodicity of Algorithm \eqref{new}, compared with \textit{Browkin I} and \textit{Browkin II}.

\begin{Example}
The behaviour of the periodicity of the three algorithms can be different. In this example we see some of the several cases that are possible to encounter. We consider $p=5$, but analogous observations hold for other primes. In $\mathbb{Q}_5$, $\sqrt{19}$ has a periodic continued fraction using Algorithm \eqref{new}, with expansion
\[\sqrt{19}= \left[2,\overline{-\frac{2}{5},2,\frac{1}{5},-2,-\frac{2}{5},-\frac{12}{5},\frac{2}{5},-2,\frac{8}{25},2,\frac{1}{5},-1,-\frac{2}{5},-\frac{8}{5},\frac{2}{5},-2,\frac{12}{25},2,\frac{2}{5},-1} \right],
\]
but no period has been detected with \textit{Browkin I} nor \textit{Browkin II}, up to the 1000th partial quotient. Vice versa, using \textit{Browkin II}, $\sqrt{69}$ is periodic (see above) while no period has been observed with the other algorithms. Moreover, on some square roots \textit{Browkin II} and Algorithm \eqref{new} are both periodic but they have different expansions. For example, the \textit{Browkin II} expression of $\sqrt{129}$ is
\begin{align*}
\sqrt{129}=& \Bigg[2,\frac{4}{125},-1,-\frac{4}{5},1,\frac{4}{5},-1,\frac{4}{5},-1,-\frac{2}{5},\\
&2,-\frac{1}{5},2,\frac{2}{5},2,-\frac{1}{5},2,-\frac{2}{5},\overline{-1,\frac{4}{5},-1,-\frac{1}{5},-2,\frac{3}{5}} \Bigg],  
\end{align*}
while with Algorithm \eqref{new} it is
\[\sqrt{129}= \left[2,\overline{\frac{4}{125},-1,\frac{1}{5},-\frac{4}{5},-\frac{1}{5},-1,-\frac{2}{5},2,-\frac{1}{5},2,\frac{2}{5},2,-\frac{1}{5},2,-\frac{2}{5},-1,-\frac{1}{5},-\frac{4}{5},\frac{1}{5},-1} \right].\]
Using \textit{Browkin I}, we did not detect any period for $\sqrt{129}$.
\end{Example}

In Tables \ref{Tab: Br1}, \ref{Tab: Br2}, \ref{Tab: New} in Appendix \ref{App: A}, we can see that periodic square roots of integers with \textit{Browkin I} are very few with respect to \textit{Browkin II} and Algorithm \eqref{new}. Moreover, except for $p=3$, the number of periodic expansion with Algorithm \eqref{new} is greater or equal than \textit{Browkin II}.
In Figure \ref{Fig: number} we plot the number of periodic square roots of integers for all the three algorithms, varying the prime $p$.
\begin{figure}[h]
\centering
\includegraphics[scale=0.7]{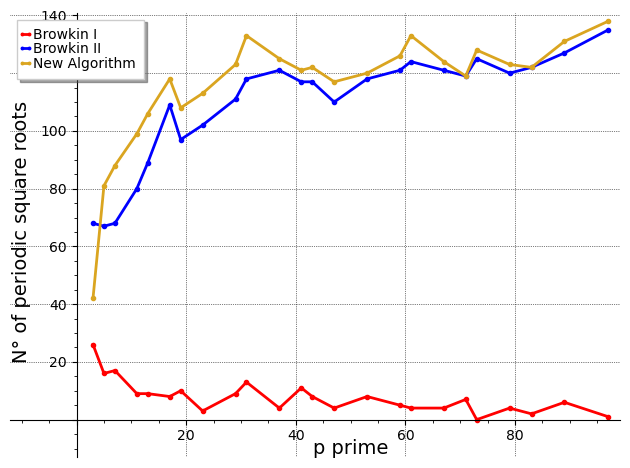}
\caption{\label{Fig: number}Number of periodic square roots of integers with \textit{Browkin I}, \textit{Browkin II} and Algorithm \eqref{new}.}
\end{figure}

\begin{Remark}\label{Rem: exp}
\textit{Browkin II} and Algorithm \eqref{new} tend to present a similar behaviour for large primes. This result was expected since the two algorithms are different only in the case of $a_0=0$ in the $p$--adic expansion of one of the odd complete quotients and the probability of having it equal to zero is around $\frac{1}{p}$, which approaches zero for growing $p$.
%Moreover, as we underline in Section \ref{Sec: periods}, short periods are very frequent also for large values of $D$ and $p$. Hence, for large $p$, it is very likely to detect periodicity before the sign function is used, i.e. before two algorithms differ on some partial quotients.
\end{Remark}

\subsection{Pre-periods of periodic expansions}
\label{Sec: preperiods}
The length of the pre-period of periodic continued fractions obtained by \textit{Browkin I} is $2$ (see \cite{BEI}) and by Algorithm \eqref{new} is $1$ (see Proposition \ref{newprep}). On the contrary, for the lengths of the pre-periods of \textit{Browkin II}, Proposition \ref{prepMRS} seems the best we can obtain. In fact, during our analysis, although most of the square roots presented pre-period $1$, we also observed pre-periods of several even lengths. In \textit{Browkin II}, pre-periods of even length are an ``anomalous" behaviour which occurs when the sign function is used in \eqref{Br2} (for more details, see \cite{MRSI}). Therefore, in light of Remark \ref{Rem: exp}, for large values of $p$ we expect to have often pre-period $1$. Indeed, for $p\geq 31$, no pre-period greater than $1$ has been observed. In Table \ref{Tab: prep}, we list the mean pre-periods of periodic \textit{Browkin II} continued fractions up to $p=29$.\\ 

\begin{table}[H]
\centering
\begin{tabular}{|l|l|l|l|l|l|l|l|l|l|}
\hline 
  \ \ \ \ \ p  & 3 & 5 & 7 & 11 & 13 & 17 & 19 & 23 & 29\\
  \hline
  Mean pre-period & 18.85 &  7.49  & 2.96 & 1.45 & 1.39 & 1.17 & 1.08 & 1.13 & 1.01 \\
  \hline
\end{tabular}
\caption{\label{Tab: prep}Mean pre-periods of periodic \textit{Browkin II} expansions.}
\end{table}

\subsection{Periods of periodic expansions}
\label{Sec: periods}

The length of the periods for periodic \textit{Browkin I} continued fractions is very often $2$, especially for large values of $p$. In fact, when $p$ increases, periodic \textit{Browkin I} continued fractions are rare and most expansions have both pre-period and period of length $2$. Moreover, let us notice that the mean period length for \textit{Browkin II} and Algorithm \eqref{new} is, in general, decreasing for growing $p$. In Figure \ref{Fig: meanperiods}, we plot the mean period of periodic square roots of integers for all the three algorithms, varying the prime $p$.
\begin{figure}[h]
\centering
\includegraphics[scale=0.75]{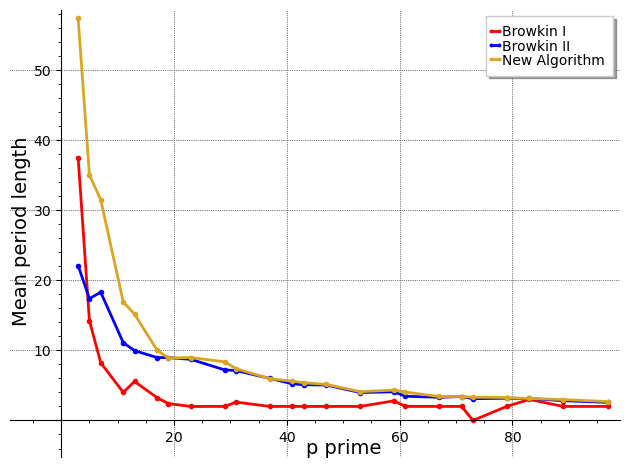}
\caption{\label{Fig: meanperiods}Mean periods of periodic square roots with \textit{Browkin I}, \textit{Browkin II} and Algorithm \eqref{new}.}
\end{figure}

From Tables \ref{Tab: Br1}, \ref{Tab: Br2}, \ref{Tab: New} in Appendix \ref{App: A}, we can observe that, when periodicity is detected, long periods are very uncommon for all values of $p$, especially for $p$ large. 
Indeed, for \textit{Browkin II} and  Algorithm \eqref{new} the $90\%$ of the periods are shorter than 30 for all $11 \leq p \leq 97$. 
In Figure \ref{Fig: periods}, we plot the length of the periods for periodic square roots of integers, in function of the size of the integer $D$, for $p = 5$.

\begin{figure}[H]\label{Fig: periodic}
\centering
\includegraphics[scale=0.75]{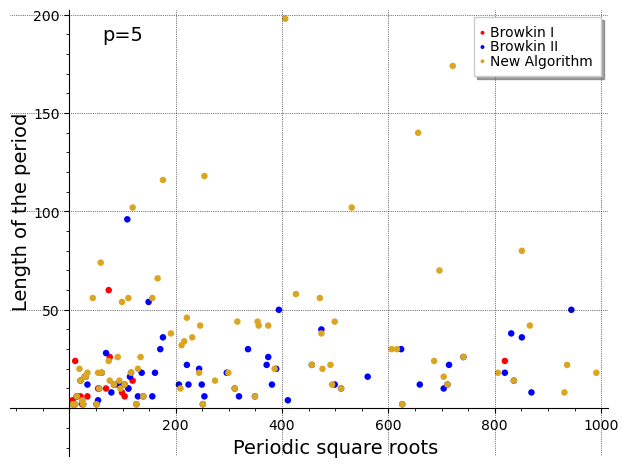}
\caption{\label{Fig: periods} Period lengths of periodic square roots with \textit{Browkin I}, \textit{Browkin II} and Algorithm \eqref{new}, for $p=5$.}
\end{figure}

\begin{Remark}\label{Rem: perd}
The numerical computations about periodicity give the suggestion that the Lagrange's theorem does not hold for $p$--adic continued fractions obtained by the studied algorithms (even if a proof of this statement is still missing). Indeed, for continued fractions over the real numbers, the maximum length of the periods for the square root $\sqrt{D}$ for all integers $1 \leq D \leq 1000$ is 60. On the contrary, for \textit{Browkin I} and Algorithm \eqref{new} there are some square roots whose expansion, if periodic, should have a period greater than 998 and 999, respectively, and this seems unlikely. A similar observation holds also for \textit{Browkin II}, even if in this case we do not know a-priori the length of the pre-period, making more difficult to deal with the study of perodicity of this algorithm. However, it seems very unlikely, for instance, that \textit{Browkin II} produces a periodic expansion for $\sqrt{19}$ in $\mathbb Q_5$ with the sum of the lengths of pre-period and period greater than 1000. 

%From Figure \ref{Fig: periods}, we can get an idea of why it is largely believed that none of the actual algorithms becomes eventually periodic for every quadratic irrationals, i.e. Lagrange's theorem does not hold for $p$--adic continued fractions up to knowledge. In fact, although we checked the expansions for 1000 complete quotients, no period greater than 100 has been detected for \textit{Browkin II} and greater than 200 for Algorithm \eqref{new}. This fact does not seem to change neither increasing $D$ nor $p$, i.e. it does not seem that bigger square roots, whenever they are periodic, show longer periods. For example, it is very unlikely that $\sqrt{19}$ has a periodic \textit{Browkin II} continued fraction in $\mathbb{Q}_5$ with a period that is greater than 1000.
\end{Remark}

\subsection{Quality of approximation}
In this section we analyze the approximations of square roots of integers by means of the sequence of convergents of the three algorithms. In general, we can observe that given $\alpha = [b_0, b_1, \ldots]$, we have
\[\alpha=[b_0,b_1,\ldots,b_n,\alpha_{n+1}]=\frac{\alpha_{n+1}A_n+A_{n-1}}{\alpha_{n+1}B_n+B_{n-1}},\]
from which
\[\alpha-\frac{A_n}{B_n}=\frac{\alpha_{n+1}A_n+A_{n-1}}{\alpha_{n+1}B_n+B_{n-1}}-\frac{A_n}{B_n}=\frac{(-1)^n}{(\alpha_{n+1}B_n+B_{n-1})B_n}.\]
Since $v_p(\alpha_{n+1})=v_p(b_{n+1})$, then $v_p(\alpha_{n+1}B_n+B_{n-1})=v_p(B_{n+1})$ and therefore
\begin{equation}\label{Eq: val2}
v_p\left(\alpha_0-\frac{A_n}{B_n}\right)=-v_p(B_{n}B_{n+1}),
\end{equation}
i.e., 
\begin{equation}\label{Eq: val}
\left|\alpha-\frac{A_n}{B_n}\right|_p=p^{v_p(B_nB_{n+1})}.
\end{equation}
It can be also proved by induction that
\[v_p(B_n)=v_p(b_0)+v_p(b_1)+\ldots+v_p(b_n),\]
for all $n\in\mathbb{N}$. Thus, considering \eqref{Eq: val2} and \eqref{Eq: val}, the study of the quality of the approximations of $\alpha$ is related to the decreasing of $v_p(B_n)$ and consequently to the values of $v_p(b_n)$. From the definitions of the three algorithms, we know that:
\begin{enumerate}
    \item[i)] $v_p(b_n)<0$ for all $n\in\mathbb{N}$, for \textit{Browkin I}
    \item[ii)] $v_p(b_{2n})=0$ and $v_p(b_{2n+1})<0$ for all $n\in\mathbb{N}$, for \textit{Browkin II}
    \item[iii)] $v_p(b_{2n})\leq 0$ and $v_p(b_{2n+1})<0$ for all $n\in\mathbb{N}$, for Algorithm \eqref{new}.
\end{enumerate}
Therefore, we expect the approximations given by the convergents of continued fractions obtained by \textit{Browkin I} to be better than those obtained by Algorithm \eqref{new}, which should be better than \textit{Browkin II}.
In Table \ref{Tab: approx} we list, for some values of $p$, the mean valuation after $10$, $100$ and $1000$ steps through \textit{Browkin I}, \textit{Browkin II} and Algorithm \eqref{new}.

\begin{table}[H]
\centering
\begin{tabular}{|l|l|l|l|}
\hline 
  p=5  & 10 steps  & 100 steps & 1000 steps \\
  \hline
\ \textit{Browkin I} & \ \ -11.1 & \ \   -123.1 & \ \ -1246.4  \\
\hline
\ \textit{Browkin II} & \ \ -6.1 & \ \  -62.8 & \ \ -629.8 \\
\hline
\ Algorithm \eqref{new} & \ \ -7.0 & \ \ -73.6  &\ \ -744.3  \\
\hline
\end{tabular}
\ \\ \ \\ \ \\
\begin{tabular}{|l|l|l|l|}
\hline 
  p=23  & 10 steps  & 100 steps & 1000 steps \\
  \hline
\ \textit{Browkin I} & \ \ -9.4 &\ \ -103.4 & \ \ -1043.8  \\
\hline
\ \textit{Browkin II} & \ \ -5.2 &\ \  -52.6 & \ \ -525.8 \\
\hline
\ Algorithm \eqref{new} & \ \ -5.4 &\ \  -54.7  &\ \ -547.8  \\
\hline
\end{tabular}
\ \\ \ \\ \ \\
\begin{tabular}{|l|l|l|l|}
\hline 
  p=47  & 10 steps  & 100 steps & 1000 steps \\
  \hline
\ \textit{Browkin I} & \ \ -9.2 &\ \  -101.2 & \ -1021.6  \\
\hline
\ \textit{Browkin II} & \ \ -5.0 &\ \  -50.9 & \ -508.9 \\
\hline
\ Algorithm \eqref{new} & \ \ -5.1 &\ \ -52.0  &\ -520.5  \\
\hline
\end{tabular}
\ \\ \ \\ \ \\
\begin{tabular}{|l|l|l|l|}
\hline 
  p=89  & 10 steps  & 100 steps & 1000 steps \\
  \hline
\ \textit{Browkin I} & \ \  -9.1 &\ \  -100.1 & \ -1010.3  \\
\hline
\ \textit{Browkin II} & \ \  -5.0 &\ \   -50.5 & \ -504.3 \\
\hline
\ Algorithm \eqref{new} & \ \  -5.1 &\ \  -50.9  &\ -509.2 \\
\hline
\end{tabular}
\ \\ 
\caption{\label{Tab: approx}Values of $v_p(B_n)$ with \textit{Browkin I}, \textit{Browkin II} and Algorithm \eqref{new} after 10, 100 and 1000 steps for $p=5,23,47,89$.}
\end{table}

The experimental results listed in the previous tables are in line with the considerations on the valuation of the partial quotients of the three algorithms. In fact, \textit{Browkin I} decreases the valuation of $B_n$ at each step, \textit{Browkin II} at half of the steps and Algorithm \eqref{new} on slightly more than half of the steps. Let us compare the quality of this approximation, given by the convergents of a continued fraction, with the classical rational approximation given by the $p$--adic expansion of $\alpha$ stopped at the $n$-th term. For  $\alpha=\sum\limits_{i=0}^{+\infty}a_ip^i\in\mathbb{Q}_p$, the sequence $\{C_n\}_{n\in\mathbb{N}}$, with $C_n=a_0+a_1p+\ldots+a_np^n$, approximates $\alpha$ with error
\[\left|\alpha-C_n\right|_p=|a_{n+1}p^{n+1}+\ldots|_p\leq\frac{1}{p^{n+1}}.\]
Therefore $C_n$ usually provides a better approximation than \textit{Browkin II} and Algorithm \eqref{new} but worse than \textit{Browkin I}.

\section{Conclusions and further research}
In this paper we have defined a new algorithm, obtained as a small modification of the $p$--adic continued fraction algorithm presented in \cite{BII}, namely \textit{Browkin II}. 
%The latter algorithm, \textit{Browkin II}, is considered the best one at the state of the art, since it is the nearest to a $p$--adic analogue of Lagrange's Theorem. 
Algorithm \eqref{new} improves the properties of periodicity of \textit{Browkin II} both in theoretical and experimental results. In particular, for $p$-adic continued fractions obtained by Algorithm \eqref{new}, an analogue of the Galois' Theorem holds and the pre--period for the expansion of square root of integers is always of length 1, like in the real case. Moreover, the numerical computations suggest that Algorithm \eqref{new} provides more periodic expansion for quadratic irrationals than \textit{Browkin II}.
Hence, it turns out that the sign function in \eqref{Br2} affects only negatively the behaviour of this algorithm. Moreover, as highlighted in Remark \ref{Rem: exp}, the two algorithm tend to be similar for large values of $p$, where the sign function is rarely used. Therefore, it could be abandoned without, apparently, any negative effect. However, the problem of finding a $p$--adic algorithm which becomes eventually periodic on every quadratic irrational still remains open. In particular, effective characterizations for periodic continued fractions provided by Algorithms \eqref{Br1}, \eqref{Br2} and \eqref{new} have not been proved.
One important matter for which we do not have an answer yet is why the periodicity properties of \textit{Browkin I} improve drastically when alternating the functions $s$ and $t$ instead of using only the function $s$ (see Figure \ref{Fig: number}). 
Finally, in light of Remark \ref{Rem: perd}, it would be interesting to deepen the study of the lengths of the periods, providing some upper bounds that could be exploited for proving that an analogue of the Lagrange's Theorem does not hold for the considered algorithms. 
%In fact, long periods for $\sqrt{D}$ are very rare, also for large values of $D$ and the prime $p$.
\appendix

\section{Tables}\label{App: A}
In the following tables we collect the computational results about the periodicity properties of Algorithms \eqref{Br1},\eqref{Br2} and \eqref{new}. All the computations have been performed on the first $1000$ complete quotients of $\sqrt{D}\in\mathbb{Q}_p$, for all the odd primes $p$ less than $100$ and $1\leq D\leq 1000$, with $D$ not a square and $v_p(D)=0$. The numerical simulations have been performed in SageMath and the code is publicly available \footnote{\href{https://github.com/giulianoromeont/p-adic-continued-fractions}{https://github.com/giulianoromeont/p-adic-continued-fractions}}. The tables collect results about:
\begin{enumerate}
    \item[•] the number of square roots which are periodic within $1000$ steps,
    \item[•] the mean length of the period
    \item[•] the value $h$ such that $75\%$ of the lengths of the periods detected are less or equal $h$,
    \item[•] the value $h$ such that $90\%$ of the lengths of the periods detected are less or equal $h$,
    \item[•] the total number of positive integers $D$ less than $1000$ such that $\sqrt{D}\in\mathbb{Q}_p$, $D$ is not a square and $v_p(D)=0$.
\end{enumerate}

\newpage
\begin{table}[!ht]
\centering
\begin{tabular}{|l|l|l|l|l|l|}
\hline 
  p  & Periodic & Mean period & \ 75\% &\ 90\% & Total\\
  \hline
\ 3 & \ \ \ \  26 &\ \ \ \ \ 37.46 & \ 52 &\ 88 &\  313 \\
\hline
\ 5 & \ \ \ \  16 & \ \ \ \ \ 14.25 & \ 14 &\ 24 &\  375\\
\hline
\ 7 & \ \ \ \ 17 & \ \ \ \ \ \ 8.24 & \ \ 8 &\ 16 &\  402  \\
\hline
11 & \ \ \ \ \ 9 &\ \ \ \ \ \ \  4  &\ \  4 &\ \ 6 & \ 426     \\
\hline
13 &\ \ \ \ \  9 &\ \ \ \ \  \ 5.56  &\ \  6 &\ \  6 & \ 433      \\
\hline
17 &\ \ \ \ \  8 &\ \ \ \ \ \  3.25  &\  \ 4 &\  \ 4 & \ 440      \\
\hline
19 & \ \ \ \  10 &\ \ \ \ \ \  2.4       &\  \ 2 &\ \  2 & \ 445  \\
\hline
23 & \ \ \ \ \ 3 &\  \ \ \ \  \  \ 2      &\  \ 2 &\ \  2 & \ 450  \\
\hline
29 & \ \ \ \ \ 9  &\  \ \ \ \ \   \  2    &\  \ 2 &\ \  2 & \ 453 \\
\hline
31 & \ \ \ \  13  &\ \ \ \  \ \     2.62  &\  \ 2 &\  \ 2 & \ 456   \\
\hline
37 & \ \ \ \ \ 4  &\ \ \ \ \  \  \    2  &\  \ 2 &\ \  2 & \ 456   \\
\hline
41 & \ \ \ \  11 &\ \ \ \ \  \  \   2  &\ \  2 &\ \  2 & \ 457   \\
\hline
43 & \ \ \ \ \ 8  &\  \ \ \ \ \    \ 2 &\ \  2 &\  \ 2 & \ 458    \\
\hline
47 & \ \ \ \ \ 4  &\  \ \ \ \ \   \  2 &\ \  2 &\ \  2 & \ 461    \\
\hline
53 & \ \ \ \ \ 8  &\  \ \ \ \  \  \  2  &\ \  2 &\ \  2 & \ 460   \\
\hline
59 & \ \ \ \ \ 5  &\  \ \ \ \  \   2.8  &\  \ 2 &\  \ 2 & \ 461   \\
\hline
61 & \ \ \ \ \ 4  &\  \ \ \ \ \ \    2  &\  \ 2 &\ \  2 & \ 462   \\
\hline
67 & \ \ \ \ \ 4  &\  \ \ \ \  \  \  2 &\ \  2 &\  \ 2 & \ 462   \\
\hline
71 & \ \ \ \ \ 7  &\  \ \ \ \ \   \  2  &\  \ 2 &\  \ 2 & \ 465   \\
\hline
73 & \ \ \ \ \ 0  &\ \ \ \  \    \text{none}  & \text{none} & \text{none} & \ 462   \\
\hline
79 & \ \ \ \ \ 4  &\  \ \ \ \  \ \   2 &\  \ 2 &\ 2 & \ 468   \\
\hline
83 & \ \ \ \ \ 2  &\  \ \ \ \  \  \  3  &\ \  2 &\ 2 & \ 464   \\
\hline
89 & \ \ \ \ \ 6  &\  \ \  \ \ \  \  2  &\  \ 2 &\ 2 & \ 466   \\
\hline
97 & \ \ \ \ \ 1  &\ \ \ \ \  \  \   2  &\ \  2 &\ 2 & \ 464   \\
\hline

\end{tabular}
\caption{\label{Tab: Br1}\textit{Browkin I}}
\end{table}

\newpage
\begin{table}[!ht]
\centering
\begin{tabular}{|l|l|l|l|l|l|}
\hline 
   p  & Periodic & Mean period &\ 75\% &\ 90\% & Total\\
  \hline
\ 3 & \ \ \ \ 68 & \ \ \ \ \ 22.09 &\ 26 &\ 42 & \ 313 \\
\hline
\ 5 & \  \ \ \ 67 & \ \ \ \ \  17.37 &\ 22 &\ 36 & \ 375\\
\hline
\ 7 & \  \ \ \ 68 &\ \ \ \ \  18.29 &\ 22 &\ 42 & \ 402  \\
\hline
11 & \  \ \ \ 80 &\ \ \ \ \  11.10  &\ 16 &\ 22 & \ 426     \\
\hline
13 & \ \ \ \  89 &\ \ \ \ \ \   9.96  &\ 10 &\ 18 & \ 433      \\
\hline
17 & \ \ \  109 &\  \ \ \ \ \   8.97 &\ 10 &\ 20 & \ 440      \\
\hline
19 & \ \ \ \ 97 &\ \ \ \ \ \    8.97   &\ 10 &\ 14 & \ 445  \\
\hline
23 & \ \ \ 102 &\ \ \ \ \  \   8.70      &\ 10 &\ 20 & \ 450  \\
\hline
29 & \ \ \ 111  &\  \ \ \ \  \     7.21   &\ \ 8 &\ 14 & \ 453 \\
\hline
31 & \ \ \ 118  &\  \ \ \ \  \     7.12  &\ \  8 &\ 14 & \ 456   \\
\hline
37 & \ \ \ 121  &\ \ \ \ \   \     5.98  &\ \  6 &\ 12 & \ 456   \\
\hline
41 & \ \ \ 117 &\  \ \ \ \  \    5.23  &\ \  6 &\ 10 & \ 457   \\
\hline
43 & \ \ \ 117  &\ \ \ \ \  \     5.09 &\ \  6 &\ 10 & \ 458    \\
\hline
47 &  \ \ \ 110  &\ \ \ \ \  \    5.05 &\ \  6 &\ 10 & \ 461    \\
\hline
53 &  \ \ \ 118  &\  \ \ \ \   \    3.98 &\ \  6 &\ \  8 & \ 460   \\
\hline
59 &  \ \ \ 121  &\  \ \ \ \   \    4.08  &\ \  6 &\ \  6 & \ 461   \\
\hline
61 &  \ \ \ 124  &\  \ \ \ \   \   3.45 &\ \  4 &\ \  6 & \ 462   \\
\hline
67 &  \ \ \ 121  &\  \ \ \ \   \     3.30 &\ \  4 &\ \  6 & \ 462   \\
\hline
71 &  \ \ \ 119  &\ \ \ \ \   \     3.41  &\ \  4 &\ \  6 & \ 465   \\
\hline
73 &  \ \ \ 125  &\ \ \ \ \   \    3.10  &\ \  4 &\ \  6 & \ 462   \\
\hline
79 &  \ \ \ 120  &\ \ \ \ \    \     3.17 &\ \  2 &\ \  6 & \ 468   \\
\hline
83 &  \ \ \ 122  &\ \ \ \ \   \     3.13  &\ \  2 &\ \  6 & \ 464   \\
\hline
89 &  \ \ \ 127  &\ \ \ \ \   \     2.82  &\ \  2 &\ \  6 & \ 466   \\
\hline
97 &  \ \ \ 135  &\ \ \ \ \  \      2.58  &\ \  2 &\ \  4 & \ 464   \\
\hline

\end{tabular}
\caption{\label{Tab: Br2}\textit{Browkin II}}
\end{table}

\newpage
\begin{table}[!ht]
\centering
\begin{tabular}{|l|l|l|l|l|l|}
\hline 
   p  & Periodic & Mean period &\ 75\% &\ 90\% & Total\\
  \hline
\ 3 & \ \ \ \  42 & \ \ \ \ \ 57.38 &\ 72 &\ 112 & \ 313 \\
\hline
\ 5 & \ \ \ \  81 &\ \ \ \ \ 35.01 &\ 42 &\ 70 & \ 375\\
\hline
\ 7 & \ \ \ \  88 &\ \ \ \ \  31.50 &\ 38 &\ 80 & \ 402  \\
\hline
11 & \ \ \ \  99 &\ \ \ \ \  16.89  &\ 22 &\ 30 & \ 426     \\
\hline
13 &\ \ \  106 &\  \ \ \ \ 15.17  &\ 18 &\ 30 & \ 433      \\
\hline
17 &\ \ \   118 &\ \ \ \ \  10.02 &\ 14 &\ 22 & \ 440      \\
\hline
19 & \ \ \  108 &\ \ \ \ \ \  8.91   &\ 10 &\ 18 & \ 445  \\
\hline
23 & \ \ \  113 &\  \ \ \ \ \  8.97   &\ 10 &\ 22 & \ 450  \\
\hline
29 & \ \ \  123 &\ \ \ \ \ \    8.36   &\ 10 &\ 18 & \ 453 \\
\hline
31 & \ \ \  133  &\  \ \ \ \ \    7.38  &\  \ 8 &\ 14 & \ 456   \\
\hline
37 & \ \ \  125  &\  \ \ \ \ \    5.95  &\  \ 6 &\ 12 &\ 456   \\
\hline
41 & \ \ \ 121 &\  \ \ \ \ \     5.60  &\  \ 6 &\ 10 & \ 457   \\
\hline
43 & \ \ \  122  &\  \ \ \ \ \   5.38 &\ \  6 &\  10 & \ 458    \\
\hline
47 & \ \ \  117  &\ \ \ \ \ \    5.15 &\ \  6 &\  10 & \ 461    \\
\hline
53 & \ \ \  120  &\ \ \ \ \ \    4.10 &\ \ 6 &\  \ 8 & \ 460   \\
\hline
59 & \ \ \  126  &\  \ \ \ \ \   4.33  &\ \  6 &\   10 & \ 461   \\
\hline
61 & \ \ \  133  &\  \ \ \ \ \  4.05 &\ \  6 &\ \  8 & \ 462   \\
\hline
67 & \ \ \  124  &\  \ \ \ \ \   3.42 &\ \  4 &\ \  6 & \ 462   \\
\hline
71 & \ \ \  119  &\ \ \ \ \ \    3.41  &\ \  4 &\  \ 6 & \ 465   \\
\hline
73 & \ \ \  128  &\  \ \ \ \ \  3.31  &\ \  4 &\ \  6 & \ 462   \\
\hline
79 & \ \ \  123  &\ \ \ \ \ \   3.27 &\ \  2 &\ \  6 & \ 468   \\
\hline
83 & \ \ \  122  &\  \ \ \ \ \    3.13  &\ \  2 &\ \  6 & \ 464   \\
\hline
89 & \ \ \  131 &\ \ \ \ \ \   2.98  &\ \  2 &\ \  6 & \ 466   \\
\hline
97 & \ \ \  138  &\  \ \ \ \ \   2.70  &\ \ 2 &\ \  6 & \ 464   \\
\hline

\end{tabular}
\caption{\label{Tab: New}Algorithm \eqref{new}}
\end{table}

\end{document}